\documentclass[10pt]{amsart}
\usepackage{latexsym, amsmath, amssymb}

\newcommand{\newsection}[1]
{\subsection{#1}\setcounter{theorem}{0} \setcounter{equation}{0}
\par \noindent \vspace{5pt}}

\newcounter{theorem}

\theoremstyle{plain}
\newtheorem{thm}[theorem]{Theorem}
\newtheorem{lem}[theorem]{Lemma}

\newtheorem{prop}[theorem]{Proposition}

\theoremstyle{definition}

\theoremstyle{remark}

\newcommand{\supp}{\operatorname{supp}}

\newcommand{\R}{\mathbb R}

\newcommand{\lap}[1]{\sqrt{-\Delta_{#1}}}

\newcommand{\Loop}{{\mathcal L}}

\newcommand{\WF}{\operatorname{WF}}

\renewcommand{\epsilon}{\varepsilon}

\usepackage{tikz}
\tikzset{node distance=2.5cm, auto}

\title[integrals of eigenfunctions over submanifolds]{looping directions and integrals of eigenfunctions over submanifolds}
\author{Emmett L. Wyman}
\address{Department of Mathematics, Johns Hopkins University, Baltimore, MD 21218}

\begin{document}

\maketitle

\begin{abstract}
Let $(M,g)$ be a compact $n$-dimensional Riemannian manifold without boundary and $e_\lambda$ be an $L^2$-normalized eigenfunction of the Laplace-Beltrami operator with respect to the metric $g$, i.e
\[
	-\Delta_g e_\lambda = \lambda^2 e_\lambda \qquad \text{ and } \qquad \| e_\lambda \|_{L^2(M)} = 1.
\]
Let $\Sigma$ be a $d$-dimensional submanifold and $d\mu$ a smooth, compactly supported measure on $\Sigma$. It is well-known (e.g. proved by Zelditch in ~\cite{ZelK} in far greater generality) that
\[
	\int_\Sigma e_\lambda \, d\mu = O(\lambda^\frac{n-d-1}{2}).
\]
We show this bound improves to $o(\lambda^\frac{n-d-1}{2})$ provided the set of looping directions,
\[
	\Loop_{\Sigma} = \{ (x,\xi) \in SN^*\Sigma : \Phi_t(x,\xi) \in SN^*\Sigma \text{ for some } t > 0 \}
\]
has measure zero as a subset of $SN^*\Sigma$, where here $\Phi_t$ is the geodesic flow on the cosphere bundle $S^*M$ and $SN^*\Sigma$ is the unit conormal bundle over $\Sigma$.
\end{abstract}


\newsection{Introduction}

In what follows, $(M,g)$ will denote a compact, boundaryless, $n$-dimensional Riemannian manifold. Let $\Delta_g$ denote the Laplace-Beltrami operator and $e_\lambda$ an $L^2$-normalized eigenfunction of $\Delta_g$ on $M$, i.e.
\[
    -\Delta_g e_\lambda = \lambda^2 e_\lambda \qquad \text{ and } \qquad \| e_\lambda \|_{L^2(M)} = 1.
\]

In ~\cite{SZmanifolds}, Sogge and Zelditch investigate which manifolds have a sequence of eigenfunctions $e_\lambda$ with $\lambda \to \infty$ which saturate the bound
\[
    \|e_\lambda\|_{L^\infty(M)} = O(\lambda^\frac{n-1}{2}).
\]
They show that the bound above is necessarily $o(\lambda^\frac{n-1}{2})$ if at each each $x$, the set of looping directions through $x$,
\[
	\Loop_x = \{\xi \in S_x^* M : \Phi_t(x,\xi) \in S_x^*M \text{ for some } t > 0 \}
\]
has measure zero\footnote{Let $\psi_j : U_j \subset \R^n \to M$ be coordinate charts of a general manifold $M$. We say a set $E \subset M$ has measure zero if the preimage $\psi_j^{-1}(E)$ has Lebesgue measure $0$ in $\R^n$ for each chart $\psi_j$. Sets of Lebesgue measure zero are preserved under transition maps, ensuring this definition is intrinsic to the $C^\infty$ structure of $M$.} as a subset of $S^*_xM$ for each $x \in M$. Here, $\Phi_t$ denotes the geodesic flow on the unit cosphere bundle $S^*M$ after time $t$. The hypotheses were later weakened by Sogge, Toth, and Zelditch in ~\cite{STZ}, where they showed
\[
    \|e_\lambda\|_{L^\infty(M)} = o(\lambda^\frac{n-1}{2})
\]
provided the set of \emph{recurrent} directions at $x$ has measure zero for each $x \in M$.

We are interested in extending the result in ~\cite{SZmanifolds} to integrals of eigenfunctions over submanifolds. Let $\Sigma$ be a submanifold of dimension $d$ with $d < n$ and a measure $d\mu(x) = h(x)d\sigma(x)$ where $d\sigma$ is the surface measure on $\Sigma$ and $h$ is a smooth function supported on a compact subset of $\Sigma$. In his 1992 paper ~\cite{ZelK}, Zelditch proves, among other things, a Weyl law- type bound
\begin{equation}\label{zelditch's weyl law}
	\sum_{\lambda_j \leq \lambda} \left| \int_\Sigma e_j \, d\mu \right|^2 \sim \lambda^{n-d} + O(\lambda^{n-d-1})
\end{equation}
from which follows
\begin{equation} \label{big O}
    \int_\Sigma e_\lambda \, d\mu = O(\lambda^{\frac{n-d-1}{2}}).
\end{equation}
Though \eqref{big O} is already well known, we will give a direct proof which will be illustrative for our main argument.

\begin{thm}\label{standard theorem}
Let $\Sigma$ be a $d$-dimensional submanifold with $0 \leq d < n$, and $d\mu(x) = h(x) d\sigma(x)$ where $h$ is a smooth, real valued function supported on a compact neighborhood in $\Sigma$. Then, \eqref{big O} holds.
\end{thm}

We let $SN^*\Sigma$ denote the unit conormal bundle over $\Sigma$. We define the set of looping directions through $\Sigma$ by
\[
	\Loop_\Sigma = \{ (x,\xi) \in SN^* \Sigma : \Phi_t(x,\xi) \in SN^*\Sigma \text{ for some } t > 0 \}.
\]
Our main result shows the bound \eqref{big O} cannot be saturated whenever the set of looping directions through $\Sigma$ has measure zero.

\begin{thm} \label{looping theorem}
    Assume the hypotheses of Theorem \ref{standard theorem} and additionally that $\Loop_{\Sigma}$ has measure zero as a subset of $SN^*\Sigma$. Then,
    \[
        \int_\Sigma e_\lambda \, d\mu = o(\lambda^\frac{n-d-1}{2}).
    \]
\end{thm}

The argument for Theorem \ref{looping theorem} is modeled after Sogge and Zelditch's arguments in ~\cite{SZmanifolds}. In fact if $d = 0$ we obtain the first part of ~\cite[Theorem 1.2]{SZmanifolds}.

We expect the bound \eqref{big O} to be saturated in the case $M = S^n$, since $\Loop_\Sigma = SN^*\Sigma$ always. The spectrum of $-\Delta_g$ on $S^n$ consists of $\lambda_j^2$ where
\[
	\lambda_j = \sqrt{j(j+n-1)} \qquad \text{ for } j = 0,1,2,\ldots
\]
(see ~\cite{SoggeHang}). For each $\lambda_j$ we select an eigenfunction $e_j$ maximizing $\left| \int_\Sigma e_j \, d\mu \right|$.
By Zelditch's Weyl law type bound \eqref{zelditch's weyl law}, there exists an increasing sequence of $\lambda$ with $\lambda \to \infty$ for which
\[
	\sum_{\lambda_j \in [\lambda, \lambda+1]} \left| \int_\Sigma e_j \, d\mu \right|^2 \gtrsim \lambda^{n-d-1}.
\]
Since the gaps $\lambda_j - \lambda_{j-1}$ approach a constant width of $1$ as $j \to \infty$, we may pick a subsequence of $\lambda$'s so that only one $\lambda_j$ falls in each band $[\lambda,\lambda+1]$. Hence,
\[
	\left|\int_\Sigma e_j \, d\mu \right| \gtrsim \lambda_j^\frac{n-d-1}{2}
\]
for some subsequence of $\lambda_j$.

It is worth remarking that there are some cases where the hypotheses of Theorem \ref{looping theorem} are naturally fulfilled and we obtain an improvement over \eqref{big O}. Chen and Sogge ~\cite{CS} proved that if $M$ is $2$-dimensional and has negative sectional curvature, and $\Sigma$ is a geodesic in $M$,
\[
	\int_\Sigma e_\lambda \, d\mu = o(1).
\]
They consider a lift $\tilde \Sigma$ of $\Sigma$ to the universal cover of $M$. Using the Gauss-Bonnet theorem, they show for each non-identity deck transformation $\alpha$, there is at most one geodesic which intersects both $\tilde \Sigma$ and $\alpha(\tilde \Sigma)$ perpendicularly. Since there are only countably many deck transformations, $\Loop_\Sigma$ is at most a countable subset of $SN^*\Sigma$ and so satisfies the hypotheses of Theorem \ref{looping theorem}. This result was extended to a larger class of curves in ~\cite{W} which similarly have countable $\Loop_\Sigma$. \\

\noindent \textbf{Acknowledgements.} The author would like to thank Yakun Xi for pointing out an error at the end of the proof of Proposition \ref{standard prop}.


\newsection{Proof of Theorem \ref{standard theorem}}

Theorem \ref{standard theorem} is a consequence of this stronger result.

\begin{prop}\label{standard prop}
Given the hypotheses of Theorem \ref{standard theorem}, we have
\[
	\sum_{\lambda_j \in [\lambda, \lambda + 1]} \left| \int_\Sigma e_\lambda \, d\mu \right|^2 \leq C \lambda^{n-d-1}.
\]
\end{prop}


We lay out some local coordinates which we will use repeatedly. Fix $p \in \Sigma$, and consider local coordinates $x = (x_1,\ldots, x_n) = (x',\bar x)$ centered about $p$, where $x'$ denotes the first $d$ coordinates and $\bar x$ the remaining $n-d$ coordinates. We let $(x',0)$ parametrize $\Sigma$ on a neighborhood of $p$ in such a way that $dx'$ agrees with the surface measure on $\Sigma$. Let $g$ denote the metric tensor with respect to our local coordinates. We require
\[
    g = \left[ \begin{array}{cc}
    * & 0 \\
    0 & I
    \end{array}
    \right] \qquad \text{ wherever } \bar x = 0,
\]
where $I$ here is the $(n-d) \times (n-d)$ identity matrix. This is ensured after inductively picking smooth sections $v_j(x')$ of $SN\Sigma$ for $j = d+1,\ldots, n$ with $\langle v_i, v_j \rangle = \delta_{ij}$, and then using
\begin{equation} \label{coordinates}
    (x_1,\ldots, x_n) \mapsto \exp(x_{d+1} v_{d+1}(x') + \cdots + x_{n} v_{n}(x'))
\end{equation}
as our coordinate map.


Now we prove Proposition \ref{standard prop}. For simplicity, we assume without loss of generality that $d\mu$ is a real measure. We set\footnote{This reduction is standard and appears in ~\cite{SZmanifolds}, ~\cite{CS}, proofs of the sharp Weyl law as presented in ~\cite{SoggeHang} and ~\cite{FIOs}, and in many other similar problems.} $\chi \in C^\infty(\R)$ with $\chi \geq 0$ and $\hat \chi$ supported on a small neighborhood of $0$. It suffices to show
\[
	\sum_j \chi(\lambda_j - \lambda) \left| \int_\Sigma e_j \, d\mu \right|^2 \leq C\lambda^{n-d-1}.
\]
By Fourier inversion, we write the left hand side as
\begin{align}
\nonumber \sum_j \int_\Sigma \int_\Sigma &\chi(\lambda_j - \lambda) e_j(x) \overline{e_j(y)} \, d\mu(x) \, d\mu(y) \\
\nonumber &= \frac{1}{2\pi} \sum_j \int_{-\infty}^\infty \int_\Sigma \int_\Sigma \hat \chi(t) e^{-it\lambda}  e^{it \lambda_j} e_j(x) \overline{e_j(y)} \, d\mu(x) \, d\mu(y) \, dt \\
\label{sp before para} &= \frac{1}{2\pi} \int_{-\infty}^\infty \int_\Sigma \int_\Sigma \hat \chi(t) e^{-it\lambda}  e^{it\lap g}(x,y) \, d\mu(x) \, d\mu(y) \, dt
\end{align}
where $e^{it\lap g}$ is the half wave operator with kernel
\[
	e^{it\lap g}(x,y) = \sum_j e^{it\lambda_j} e_j(x) \overline{e_j(y)}.
\]
Using the coordinates $x = (x', \bar x)$ as in \eqref{coordinates}, the last line of \eqref{sp before para} is written
\begin{equation} \label{sp in local coordinates}
	= \frac{1}{2\pi} \int_{-\infty}^\infty \int_{\R^d} \int_{\R^d} \hat \chi(t) e^{-it\lambda} e^{it\lap g}(x',y') h(x') h(y') \, dx' \, dy' \, dt
\end{equation}
where $h$ is a smooth function on $\R^{d}$ such that $d\mu(x) = h(x')dx'$, and where by abuse of notation $x'$ is taken to mean $(x',0)$ where appropriate. We now use H\"ormander's parametrix as presented in ~\cite{FIOs}, i.e
\[
	e^{it\lap g}(x,y) = \frac{1}{(2\pi)^n} \int_{\R^n} e^{i(\varphi(x,y,\xi) + tp(y,\xi))} q(t,x,y,\xi) \, d\xi
\]
modulo a smooth kernel, where
\[
	p(y,\xi) = \sqrt{\sum_{j,k} g^{jk}(y) \xi_j \xi_k}
\]
is the principal symbol of $\lap g$ and $\varphi$ is smooth for $|\xi| > 0$, homogeneous of degree $1$ in $\xi$, and satisfies
\begin{equation} \label{local varphi}
	|\partial_{\xi}^\alpha (\varphi(x,y,\xi) - \langle x - y, \xi \rangle)| \leq C_\alpha |x - y|^2|\xi|^{1 - |\alpha|}
\end{equation}
for multiindices $\alpha \geq 0$ and for $x$ and $y$ sufficiently close.
Moreover, $q$ satisfies bounds
\begin{equation} \label{q bounds}
	|\partial_\xi^\alpha \partial_{t,x,y}^\beta q(t,x,y,\xi)| \leq C_{\alpha,\beta} (1 + |\xi|)^{-|\alpha|},
\end{equation}
and where for $t \in \supp \hat \chi$, $q$ is supported on a small neighborhood of $x = y$. Hence, we write \eqref{sp in local coordinates} as
\begin{align*}
	= \frac{1}{(2\pi)^{n+1}} \int_{-\infty}^\infty \int_{\R^d} \int_{\R^d} \int_{\R^n} e^{i(\varphi(x',y',\xi) + tp(y',\xi) - t\lambda)} \hat \chi(t) q(t,x',y',\xi) &h(x')h(y')\\
	&\, d\xi \, dx' \, dy' \, dt,
\end{align*}
and after making a change of coordinates $\xi \mapsto \lambda \xi$ is
\begin{align*}
	= \frac{\lambda^n}{(2\pi)^{n+1}} \int_{-\infty}^\infty \int_{\R^d} \int_{\R^d} \int_{\R^n} e^{i\lambda(\varphi(x',y',\xi) + t(p(y',\xi) - 1))} \hat \chi(t) q(t,x',y',\lambda \xi) &h(x')h(y')\\
	&\, d\xi \, dx' \, dy' \, dt.
\end{align*}
We introduce a function $\beta \in C_0^\infty(\R)$ with $\beta \equiv 1$ near $0$ and support contained in a small neighborhood of $0$, and cut the integral into $\beta(\log p(y',\xi))$ and $1 - \beta(\log p(y',\xi))$ parts. $|p(y',\xi) - 1|$ is bounded away from $0$ on the support of $1 - \beta(\log p(y',\xi))$, so integrating by parts in $t$ yields
\begin{align*}
	\frac{\lambda^n}{(2\pi)^{n+1}} \int_{-\infty}^\infty &\int_{\R^d} \int_{\R^d} \int_{\R^n} e^{i\lambda(\varphi(x',y',\xi) + t(p(y',\xi) - 1))}\\
	&(1 - \beta(\log p(y',\xi))) \hat \chi(t) q(t,x',y',\lambda \xi) h(x')h(y') \, d\xi \, dx' \, dy' \, dt.\\
	&\hspace{20em} =  O(\lambda^{-N})
\end{align*}
for each $N = 1,2,\ldots$. What is left to bound is the $\beta(\log p(y',\xi))$ part, i.e.
\begin{align} \label{sp amp and phase}
	\lambda^n \int_{-\infty}^\infty \int_{\R^d} \int_{\R^d} \int_{\R^n} e^{i \lambda \Phi(t,x',y',\xi)} a(\lambda; t, x', y', \xi) \, d\xi \, dx' \, dy' \, dt = O(\lambda^{n-d-1})
\end{align}
where we have set the amplitude
\[
	a(\lambda; t,x',y',\xi) = \frac{1}{(2\pi)^{n+1}} \beta(\log|\xi|) \hat \chi(t) q(t,x',y',\lambda \xi) h(x')h(y')
\]
and the phase
\[
	\Phi(t,x',y',\xi) = \varphi(x',y',\xi) + t(p(y',\xi) - 1).
\]
By \eqref{q bounds} and since $a$ has compact support in $t$, $x'$, $y'$, and $\xi$, $a$ and all of its derivatives are uniformly bounded in $\lambda$.

We are now in a position to apply stationary phase. Write $\xi = (\xi', \bar \xi)$ and write $\bar \xi = r \omega$ in polar coordinates with $r \geq 0$ and $\omega \in S^{n-d-1}$. The integral in \eqref{sp amp and phase} is then written
\begin{align*}
	\lambda^n \int_{-\infty}^\infty \int_{\R^d} \int_{\R^d} \int_{\R^d} \int_{S^{n-d-1}} \int_0^\infty e^{i \lambda \Phi(t,x',y',\xi)} &a(\lambda; t, x', y', \xi)\\
	 &r^{n-d-1} \, dr \, d\omega \, d\xi' \, dx' \, dy' \, dt
\end{align*}
We will fix $y'$ and $\omega$ and use the method of stationary phase in the remaining variables $t$, $x'$, $\xi'$, and $r$ (a total of $2d + 2$ dimensions). We assert that, for fixed $y'$ and $\omega$, there is a nondegenerate stationary point at $(t,x',\xi',r) = (0,y',0,1)$. $\Phi = 0$ at such a stationary point, and after perhaps shrinking the support of $a$ we apply ~\cite[Corollary 1.1.8]{FIOs} to write the left hand side of \eqref{sp amp and phase} as
\[
	\lambda^{n-d-1} \int_{\R^d} \int_{S^{n-d-1}} a(\lambda; y', \omega) \, dy' \, d\omega
\]
for some amplitude $a(\lambda;y',\omega)$ uniformly bounded with respect to $\lambda$. \eqref{sp amp and phase} follows.

We have
\begin{align}
	\nonumber \label{nabla phi} \partial_t \Phi &= p(y',\xi) - 1\\
	\nonumber \nabla_{x'} \Phi &= \nabla_{x'} \varphi(x',y',\xi)\\
	\nonumber \nabla_{\xi'} \Phi &= \nabla_{\xi'} \varphi(x',y',\xi) + t \nabla_{\xi'}p(y',\xi)\\
	\nonumber \partial_r \Phi &= \partial_r \varphi(x',y',\xi) + t \partial_r p(y',\xi).
\end{align}
Note for fixed $y'$ and $\omega$, $(t,x',\xi',r) = (0,y',0,1)$ is a critical point of $\Phi$. Now we compute the second derivatives at this point. We immediately see that $\partial_t^2 \Phi$, $\partial_t \nabla_{x'}\Phi$, $\nabla_{\xi'}^2 \Phi$, $\partial_r \nabla_{\xi'} \Phi$, and $\partial_r^2 \Phi$ all vanish. Moreover, $\partial_r \partial_r \Phi = 1$ since $p(y',\xi) = r$ where $\xi' = 0$. By our coordinates \eqref{coordinates} and the fact that $[g^{ij}]_{i,j \leq d}$ is necessarily positive definite,
\[
	p(y',\xi) = \sqrt{\sum_{j,k} g^{jk} \xi_j \xi_k} = \sqrt{r^2 + \sum_{j,k \leq d} g^{jk} \xi_j' \xi_k' } \geq r = p(y',r\omega).
\]
Hence, $\partial_t \nabla_{\xi'} \Phi = \nabla_{\xi'} p(y',\xi) = 0$. Since $\varphi$ is homogeneous of degree $1$ in $\xi$, at $\xi' = 0$ and $t = 0$,
\[
	\nabla_{x'}\partial_r \Phi = \nabla_{x'} \partial_r \varphi(x',y',\xi) = \nabla_{x'} \varphi(x',y',\omega) = 0 
\]
since $\varphi(x',y',\omega) = O(|x' - y'|^2)$ by \eqref{local varphi} and the fact that $\langle x' - y', \omega \rangle = 0$. Finally by \eqref{local varphi},
\[
	\nabla_{\xi'} \varphi(x',y',\xi' + \omega) = x' + O(|x' - y'|^2)
\]
whence at the critical point
\[
	\nabla_{x'} \nabla_{\xi'} \Phi = I,
\]
the $d \times d$ identity matrix. In summary, the Hessian matrix of $\Phi$ at the critical point $(t,x',\xi',r) = (0,y',0,1)$ is
\[
	\nabla_{t,x',\xi',r}^2 \Phi = \left[ \begin{array}{cccc}
		0 & 0 & 0 & 1 \\
		0 & * & I & 0 \\
		0 & I & 0 & 0 \\
		1 & 0 & 0 & 0
	\end{array} \right]
\]
which has full rank.


\newsection{Microlocal tools}

The hypotheses on the looping directions in Theorem \ref{looping theorem} ensure that the wavefront sets of $\mu$ and $e^{it\lap g} \mu$ have minimal intersection for any given $t$. We can then use pseudodifferential operators to break $\mu$ into two parts, the first whose wavefront set is disjoint from that of $e^{it\lap g}\mu$ and the second which contributes a small, controllable term to the bound. The following propositions will allow us to handle these cases, respectively.


\begin{prop} \label{B contribution}
    Let $u$ and $v$ be distributions on $M$ for which
    \[
        \WF(u) \cap \WF(v) = \emptyset.
    \]
    Then
    \[
    	t \mapsto \int_M e^{it\lap g} u(x) \overline{v(x)} \, dx
	\]
	is a smooth function of $t$ on some neighborhood of $0$.
\end{prop}


\begin{proof}
    Using a partition of unity, we write
    \[
        I = \sum_j A_j
    \]
    modulo a smoothing operator where $A_j \in \Psi_{\text{cl}}^0(M)$ with essential supports in small conic neighborhoods. We then write, formally,
    \begin{align*}
        \int e^{it \lap g} u(x) \overline{v(x)} \, dx = \sum_{j,k} \int A_je^{it\lap g}u(x) \overline{A_k v(x)} \, dx.
    \end{align*}
    We are done if for each $i$ and $j$,
    \begin{equation}\label{jk smooth part}
    	\int_M A_j e^{it\lap g} u(x) \overline{A_k v(x)} \, dx \qquad \text{ is smooth for } |t| \ll 1.
    \end{equation}
    If the essential supports of $A_j$ and $A_k$ are disjoint, then $A_j^* A_k$ is a smoothing operator, and so $A_j^* A_k v$ is a smooth function and the contributing term
    \[
        \int u(x) \overline{e^{it\lap g} A_j^* A_k v(x)} \, dx
    \]
    is smooth is $t$. Assume the essential support of $A_j$ are small enough so that for each $j$ there exists a small conic neighborhood $\Gamma_j$ which fully contains the essential support of $A_k$ if it intersects the essential support of $A_j$. We in turn take $\Gamma_j$ small enough so that for each $j$, $\overline{ \Gamma_j}$ either does not intersect $\WF(u)$ or does not intersect $\WF(v)$. In the latter case, $A_k v$ is smooth and we have \eqref{jk smooth part} as before. In the former case,
    \[
        \overline{\Gamma_j} \cap \WF(e^{it\lap g}u) = \emptyset \qquad \text{ for } |t| \ll 1
    \]
    since both sets above are closed and the geodesic flow is continuous. Then $A_j e^{it\lap g} u(x)$ is smooth as a function of $t$ and $x$, and we have \eqref{jk smooth part}.
\end{proof}


The second piece of our argument requires the following generalization of Proposition \ref{standard prop}, modeled after ~\cite[Lemma 5.2.2]{SoggeHang}. In the proof we will come to a point where it seems like we may have to perform a stationary phase argument involving an eight-by-eight Hessian matrix. Instead, we appeal to Proposition \ref{stationary phase tool} in the appendix to break the argument into two steps involving two four-by-four Hessian matrices.


\begin{prop} \label{b contribution}
Let $b(x,\xi)$ be smooth for $\xi \neq 0$ and homogeneous of degree $0$ in the $\xi$ variable. We define $b \in \Psi_{\text{cl}}^0(M)$ by
\[
	b(x,D)f(x) = \frac{1}{(2\pi)^n} \int_{\R^n} \int_{\R^n} e^{i\langle x - y , \xi \rangle} b(x,\xi) \, dy \, d\xi
\]
for $x$, $y$, and $\xi$ expressed locally according to our coordinates \eqref{coordinates}. Then,
\begin{align*}
    &\sum_{\lambda_j \in [\lambda,\lambda+1]} \left| \int_\Sigma b e_j(x) \, d\mu(x) \right|^2\\
    &\qquad \leq C \left( \int_{\R^d} \int_{S^{n-d-1}} |b(x',\omega)|^2 h(x')^2 \, d\omega \, dx' \right) \lambda^{n-d-1} + C_{b}\lambda^{n-d-2}
\end{align*}
where $C$ is a constant independent of $b$ and $\lambda$ and $C_b$ is a constant independent of $\lambda$ but which depends on $b$.
\end{prop}


\begin{proof}
	We may by a partition of unity assume that $b(x,D)$ has small $x$-support. Let $\chi$ be as in the proof of Proposition \ref{standard prop}. It suffices to show
    \begin{align*}
        \sum_{j} \int_\Sigma \int_\Sigma &\chi(\lambda_j - \lambda) b(x,D) e_j(x) \overline{b(y,D) e_j(y)} \, d\mu(x) \, d\mu(y)\\
        &\sim \left( \int_{\R^d} \int_{S^{n-d-1}} |b(y',\omega)|^2 h(y')^2 \, d\omega \, dy' \right) \lambda^{n-d-1} + O_{b}(\lambda^{n-d-2}).
    \end{align*}
    Using the same reduction as in Proposition \ref{standard prop}, the left hand side is
    \begin{equation} \label{b lemma 1}
    	\frac{1}{2\pi} \int_{-\infty}^\infty \int_\Sigma \int_\Sigma \hat \chi(t) e^{-it\lambda} b e^{it\lap g}b^* (x,y) \, d\mu(x) \, d\mu(y) \, dt.
    \end{equation}
    Set $\beta \in C_0^\infty(\R)$ with small support and where $\beta \equiv 1$ near $0$. Then,
    \begin{align}
    	\nonumber \int_{\R^d} &b(x',D)f(x')h(x') \, dx'\\
		\nonumber &= \frac{\lambda^n}{(2\pi)^n} \int_{\R^n} \int_{\R^n} \int_{\R^d} e^{i \lambda \langle x' - w, \eta \rangle} b(x',\eta) f(w) h(x') \, dx' \, dw \, d\eta\\
		\label{b} &= \frac{\lambda^n}{(2\pi)^n} \int_{\R^n} \int_{\R^n} \int_{\R^d} e^{i \lambda \langle x' - w, \eta \rangle} \beta(\log|\eta|) b(x',\eta) f(w) h(x') \, dx' \, dw \, d\eta\\
		\nonumber & \hspace{24em} + O(\lambda^{-N}),
    \end{align}
    where the second line is obtained by a change of variables $\eta \mapsto \lambda \eta$, and the third line is obtained after multiplying in the cutoff $\beta(\log|\eta|)$ and bounding the discrepancy by $O(\lambda^{-N})$ by integrating by parts in $x'$. Additionally,
    \begin{align}
    	\nonumber b^*(z,D)d\mu(z) &= \frac{1}{(2\pi)^n} \int_{\R^n} \int_{\R^d} e^{i\langle z - y', \zeta \rangle} b(y',\zeta) h(y') \, dy' \, d\zeta \\
		\nonumber &= \frac{\lambda^n}{(2\pi)^n} \int_{\R^n} \int_{\R^d} e^{i \lambda \langle z - y', \zeta \rangle} b(y',\zeta) h(y') \, dy' \, d\zeta \\
		\nonumber &= \frac{\lambda^n}{(2\pi)^n} \int_{\R^n} \int_{\R^d} e^{i \lambda \langle z - y', \zeta \rangle} \beta(\log|\zeta|) b(y',\zeta) h(y') \, dy' \, d\zeta + O(\lambda^{-N})\\
		\label{b*} &= \frac{\lambda^n}{(2\pi)^n} \int_{\R^n} \int_{\R^d} e^{i \lambda \langle z - y', \zeta \rangle} \beta(\log|\zeta|) \beta(|z - y'|) b(y',\zeta) h(y') \, dy' \, d\zeta\\
		\nonumber &\hspace{24em}+ O(\lambda^{-N}),
    \end{align}
    where the second and third lines are obtained similarly as before and the fourth line is obtained after multiplying by $\beta(\log|z - y'|)$ and integrating the remainder by parts in $\zeta$.
    Using H\"ormander's parametrix,
    \begin{align}
    	\nonumber \frac{1}{2\pi} &\int_{-\infty}^\infty \hat \chi(t) e^{-it\lambda} e^{it\lap g}(w,z) \, dt \\
		\nonumber &= \frac{1}{(2\pi)^{n+1}} \int_{-\infty}^\infty \int_{\R^n} e^{i(\varphi(w,z,\xi) + tp(z,\xi) - t\lambda)} \hat \chi(t) q(t,w,z,\xi) \, d\xi \, dt\\
		\nonumber &= \frac{\lambda^n}{(2\pi)^{n+1}} \int_{-\infty}^\infty \int_{\R^n} e^{i\lambda(\varphi(w,z,\xi) + t(p(z,\xi) - 1))} \hat \chi(t) q(t,w,z,\lambda \xi) \, d\xi \, dt\\
		\label{q} &= \frac{\lambda^n}{(2\pi)^{n+1}} \int_{-\infty}^\infty \int_{\R^n} e^{i\lambda(\varphi(w,z,\xi) + t(p(z,\xi) - 1))} \beta(\log p(z,\xi)) \hat \chi(t) q(t,w,z,\lambda \xi) \, d\xi\\
		\nonumber &\hspace{26em} + O(\lambda^{-N}).
    \end{align}
    Here the third line comes from a change of coordinates $\xi \mapsto \lambda \xi$. The fourth line follows after applying the cutoff $\beta(\log p(z,\zeta))$ and integrating the discrepancy by parts in $t$. Combining \eqref{b}, \eqref{b*}, and \eqref{q}, we write \eqref{b lemma 1} as
	\begin{align}
    \label{pre stationary phase}	\lambda^{3n} \int \cdots \int e^{i\Phi(t,x',y',w,z,\eta,\zeta,\xi)} &a(\lambda;t,x',y',w,z,\eta,\zeta,\xi)\\
	\nonumber &dx' \, dy' \, dw \, dz \, d\eta \, d\zeta \, d\xi + O(\lambda^{-N})
    \end{align}
    with amplitude
    \begin{align*}
    	a(\lambda;t,x',&y',w,z,\eta,\zeta,\xi) = \frac{1}{(2\pi)^{3n+1}} \hat \chi(t) q(t,w,z,\lambda \xi) \beta(\log p(z,\xi))\beta(\log |\eta|) \\
		&\hspace{12em} \beta(\log|\zeta|)\beta(|z - y'|) b(x',\eta) b(y',\zeta) h(x') h(y')
    \end{align*}
    and phase
    \[
    	\Phi(t,x',y',w,z,\eta,\zeta,\xi) = \langle x' - w, \eta \rangle + \varphi(w,z,\xi) + t(p(z,\xi) - 1) + \langle z - y', \zeta \rangle.
    \]
    We pause here to make a couple observations. First, $a$ has compact support in all variables, support which we may adjust to be smaller by controlling the supports of $\hat \chi$, $\beta$, $b$, and the support of $q$ near the diagonal. Second, the derivatives of $a$ are bounded independently of $\lambda \geq 1$. We are now in a position to use the method of stationary phase -- not in all variables at once, though. First, we fix $t$, $x'$, $y'$ and $\xi$, and use stationary phase in $w$, $z$, $\eta$, and $\zeta$. We have
    \begin{align*}
    	\nabla_w \Phi &= - \eta + \nabla_w \varphi(w,z,\xi) \\
		\nabla_z \Phi &= \nabla_z \varphi(w,z,\xi) + t \nabla_z p(z,\xi) + \zeta \\
		\nabla_\eta \Phi &= x' - w \\
		\nabla_\zeta \Phi &= z - y'
    \end{align*}
    which all simultaneously vanishes if and only if
	\begin{equation}\label{constraints}
		(w,z,\eta,\zeta) = (x',y', \nabla_x \varphi(x',y',\xi), -\nabla_y \varphi(x',y',\xi) - t\nabla_y p(y',\xi)).
	\end{equation}
	At such a critical point we have the Hessian matrix
	\[
		\nabla_{w,z,\eta,\zeta}^2 \Phi = \left[ \begin{array}{cccc}
			* & * & -I & 0 \\
			* & * & 0 & I \\
			-I & 0 & 0 & 0 \\
			0 & I & 0 & 0
		\end{array} \right],
	\]
	which has determinant $-1$. By Proposition \ref{stationary phase tool} in the appendix, \eqref{pre stationary phase} is equal to complex constant times
	\begin{align*}
		\lambda^n &\int_{-\infty}^\infty \int_{\R^n} \int_{\R^d} \int_{\R^d} e^{i\lambda \Phi(t,x',y',\xi)} a(\lambda; t, x', y', \xi) \, dx' \, dy' \, d\xi' \, dt\\
		&+ \lambda^{n-1} \int_{-\infty}^\infty \int_{\R^n} \int_{\R^d} \int_{\R^d} e^{i\lambda \Phi(t,x',y',\xi)} R(\lambda; t, x', y', \xi) \, dx' \, dy' \, d\xi' \, dt + O(\lambda^{-N})
	\end{align*}
	where we have phase
	\[
		\Phi(t,x',y',\xi) = \varphi(x',y',\xi) + t(p(y',\xi) - 1),
	\]
	amplitude
	\begin{align*}
		&a(\lambda; t, x', y', \xi) = a(\lambda;t,x',y',w,z,\eta,\zeta,\xi)
	\end{align*}
	with $w,z,\eta,$ and $\zeta$ subject to the constraints \eqref{constraints}, and where $R$ is a compactly supported smooth function in $t,x',y'$, and $\xi$ whose derivatives are bounded uniformly with respect to $\lambda$. Our phase function matches that in the proof of Proposition \ref{standard prop}, and so we repeat that argument -- we write $\bar \xi = r \omega$ and fix $y'$ and $\omega$. We obtain unique nondegenerate stationary points
	\[
		(t,x',\xi',r) = (0,y',0,1).
	\]
	Now,
	\[
		a(\lambda;0,y',y', \omega) \sim |b(y',\omega)|^2 h(y')^2.
	\]
	Hence, we have
	\begin{align*}
		\lambda^n \int_{-\infty}^\infty &\int_{\R^n} \int_{\R^d} \int_{\R^d} e^{i\lambda \Phi(t,x',y',\xi)} a(\lambda; t, x', y', \xi) \, dx' \, dy' \, d\xi' \, dt\\
		&\sim \lambda^{n-d-1} \left( \int_{\R^d} \int_{S^{n-d-1}} b(y',\omega)^2 h(y')^2 \, d\omega \, dy' \right) + O(\lambda^{n-d-2})
	\end{align*}
	by Proposition \ref{stationary phase tool} as desired. The same argument applied to the remainder term gives
	\[
		\lambda^{n-1} \int_{-\infty}^\infty \int_{\R^n} \int_{\R^d} \int_{\R^d} e^{i\lambda \Phi(t,x',y',\xi)} R(\lambda; t, x', y', \xi) \, dx' \, dy' \, d\xi' \, dt = O(\lambda^{n-d-2})
	\]
	as desired.
\end{proof}


\newsection{Proof of Theorem \ref{looping theorem}}

Theorem \ref{looping theorem} follows from the following stronger statement.

\begin{prop} \label{looping prop}
Given the hypotheses of Theorem \ref{looping theorem}, we have
\[
	\sum_{\lambda_j \in [\lambda, \lambda + \epsilon]} \left| \int_\Sigma e_\lambda \, d\mu \right|^2 \leq C\epsilon \lambda^{n-d-1} + C_\epsilon \lambda^{n-d-2},
\]
where $C$ is a constant independent of $\epsilon$ and $\lambda$, and $C_\epsilon$ is a constant depending on $\epsilon$ but not $\lambda$.
\end{prop}

We make a few convenient assumptions. First, we take the injectivity radius of $M$ to be at least $1$ by scaling the metric $g$. Second, we assume the support of $d\mu$ has diameter less than $1/2$ by a partition of unity. We reserve the right to further scale the metric $g$ and restrict the support of $d\mu$ as needed, finitely many times.

As before, we set $\chi \in C^\infty(\R)$ with $\chi(0) = 1$, $\chi \geq 0$, and $\supp \hat \chi \subset [-1,1]$. It suffices to show
\[
	\sum_j \chi(T(\lambda_j - \lambda)) \left| \int_\Sigma e_\lambda \, d\mu \right|^2 \leq CT^{-1} \lambda^{n-d-1} + C_T \lambda^{n-d-2}
\]
for $T > 1$. Similar to the reduction in the proof of Proposition \ref{standard prop}, we have
\begin{align*}
	\sum_j \int_\Sigma \int_\Sigma &\chi(T(\lambda_j - \lambda)) e_j(x) \overline{e_j(y)} \, d\mu(x) \, d\mu(y)\\
	&= \frac{1}{2\pi} \sum_j \int_{-\infty}^\infty \int_\Sigma \int_\Sigma \hat \chi(t) e^{itT(\lambda_j - \lambda)} e_j(x) \overline{e_j(y)} \, d\mu(x) \, d\mu(y) \, dt\\
	&= \frac{1}{2\pi T} \sum_j \int_{-\infty}^\infty \int_\Sigma \int_\Sigma \hat \chi(t/T) e^{-it\lambda} e^{it\lambda_j} e_j(x) \overline{e_j(y)} \, d\mu(x) \, d\mu(y) \, dt\\
	&= \frac{1}{2\pi T} \int_{-\infty}^\infty \int_\Sigma \int_\Sigma \hat \chi(t/T) e^{-it\lambda} e^{it\lap g}(x,y) \, d\mu(x) \, d\mu(y) \, dt.
\end{align*}
Hence, it suffices to show
\begin{align}
	\label{lp bound 1} &\left| \int_{-\infty}^\infty \int_\Sigma \int_\Sigma \hat \chi(t/T) e^{-it\lambda} e^{it\lap g}(x,y) \, d\mu(x) \, d\mu(y) \, dt \right|\\
	\nonumber &\hspace{16em} \leq C \lambda^{n-d-1} + C_T \lambda^{n-d-2}.
\end{align}
Set $\beta \in C_0^\infty(\R)$ with $\beta(t) \equiv 1$ near $0$ and $\beta$. We cut the integral in \eqref{lp bound 1} into $\beta(t)$ and $1 - \beta(t)$ parts. Since $\beta(t)\hat \chi(t/T)$ and its derivatives are all bounded independently of $T \geq 1$,
\[
	\left| \int_{-\infty}^\infty \int_\Sigma \int_\Sigma \beta(t) \hat \chi(t/T) e^{-it\lambda} e^{it\lap g}(x,y) \, d\mu(x) \, d\mu(y) \, dt \right| \leq C \lambda^{n-d-1}
\]
by the proof of Proposition \ref{standard prop}. Hence, it suffices to show
\begin{align}
	\label{lp bound 2} &\left| \int_{-\infty}^\infty \int_\Sigma \int_\Sigma (1 - \beta(t)) \hat \chi(t/T) e^{-it\lambda} e^{it\lap g}(x,y) \, d\mu(x) \, d\mu(y) \, dt \right|\\
	\nonumber &\hspace{20em} \leq C \lambda^{n-d-1} + C_T \lambda^{n-d-2}.
\end{align}
Here we shrink the support of $\mu$ so that $\beta(d_g(x,y)) = 1$ for $x,y\in \supp \mu$. We now state and prove a useful decomposition based off of those in ~\cite{SZmanifolds}, ~\cite{STZ}, and Chapter 5 of ~\cite{SoggeHang}. We let $\Loop_\Sigma(\supp \mu,T)$ denote the subset of $\Loop_\Sigma$ relevant to the support of $\mu$ and the timespan $[1,T]$, specifically
\begin{align*}
	\Loop_\Sigma(\supp \mu,T) &= \{ (x,\xi) \in SN^*\Sigma : \Phi_t(x,\xi) = (y,\eta) \in SN^*\Sigma\\
	&\hspace{4em} \text{ for some } t \in [1,T] \text{ and where } x,y \in \supp \mu \}.
\end{align*}


\begin{lem} \label{b B decomposition}
Fix $T > 1$ and $\epsilon > 0$. There exist $b, B \in \Psi_{\text{cl}}^0(M)$ supported on a neighborhood of $\supp \mu$ with the following properties.
\begin{enumerate}
	\item $b(x,D) + B(x,D) = I$ modulo a smoothing operator on $\supp \mu$.
	\item Using coordinates \eqref{coordinates},
	\[
		\int_{\R^d} \int_{S^{n-d-1}} |b(x',\omega)|^2 \, d\omega \, dx' < \epsilon,
	\]
	where $b(x,\xi)$ is the principal symbol of $b(x,D)$.
	\item The essential support of $B(x,D)$ contains no elements of $\Loop_\Sigma(\supp \mu,T)$.
\end{enumerate}
\end{lem}


\begin{proof}
As shorthand, we write
\[
	SN^*_{\supp \mu} \Sigma = \{ (x,\xi) \in SN^* \Sigma : x \in \supp \mu \}.
\]
We first argue that $\Loop_\Sigma(\supp \mu, T)$ is closed for each $T > 1$. However, $\Loop_\Sigma(\supp \mu, T)$ is the projection of the set
\begin{align} \label{t,x,xi}
	\{ (t,x,\xi) \in [1,T] \times SN^*_{\supp \mu} \Sigma : \Phi_t(x,\xi) \in SN^*_{\supp \mu} \Sigma \}
\end{align}
onto $SN^*_{\supp \mu}\Sigma$, and since $[1,T]$ is compact it suffices to show that \eqref{t,x,xi} is closed. However, \eqref{t,x,xi} is the intersection of $[1,T] \times SN^*_{\supp \mu}\Sigma$ with the preimage of $SN^*_{\supp \mu} \Sigma$ under the continuous map
\[
	(t,x,\xi) \mapsto \Phi_t(x,\xi).
\]
Since $SN^*_{\supp \mu}\Sigma$ is closed, \eqref{t,x,xi} is closed.

Since $\Loop_\Sigma(\supp \mu, T)$ is closed and has measure zero, there is $\tilde b \in C^\infty(SN^* \Sigma)$ supported on a neighborhood of $SN^*_{\supp \mu}\Sigma$ with $0 \leq \tilde b(x,\xi) \leq 1$, $\tilde b(x,\xi) \equiv 1$ on an open neighborhood of $\Loop_\Sigma(\supp \mu,T)$, and
\[
	\int_{\R^d}\int_{S^{n-d-1}} |b(x',\omega)|^2 \, d\omega \, dx' < \epsilon.
\]
We use the coordinates in \eqref{coordinates} and define
\[
	b(x,D)f(x) = \frac{1}{(2\pi)^n} \int_{\R^n} \int_{\R^n} e^{i\langle x - y, \xi\rangle} \tilde b(x,\xi/|\xi|) f(y) \, dy \, d\xi,
\]
hence (2). We set $\psi \in C_0^\infty(\Sigma)$ to be a cutoff function supported on a neighborhood of $\supp \mu$ with $\psi \equiv 1$ on $\supp \mu$. Defining
\[
	B(x,D)f(x) = \frac{1}{(2\pi)^n} \int_{\R^n} \int_{\R^n} e^{i\langle x - y, \xi\rangle} \psi(x)(1 - \tilde b(x,\xi/|\xi|)) f(y) \, dy \, d\xi
\]
yields (1). We have (3) since the support of $1 - \tilde b(x,\xi)$ contains no elements of $\Loop_\Sigma(\supp \mu, T)$.
\end{proof}


Returning to the proof of Proposition \ref{looping prop}, let $X_{T}$ denote the function with
\[
	\hat X_{T}(t) = (1 - \beta(t)) \hat \chi(t/T),
\]
and let $X_{T,\lambda}$ denote the operator with kernel
\[
	X_{T,\lambda}(x,y) = \frac{1}{2\pi} \int_{-\infty}^\infty \hat X_T(t) e^{-it\lambda} e^{it\lap g}(x,y) \, dt.
\]
We use part (1) of Lemma \ref{b B decomposition} to write the integral in \eqref{lp bound 2} as
\begin{align*}
	\int_\Sigma \int_\Sigma X_{T,\lambda}(x,y) \, d\mu(y) \, d\mu(x)
	=& \int_\Sigma \int_\Sigma B X_{T,\lambda} B^* (x,y) \, d\mu(y) \, d\mu(x) \\
	&+ \int_\Sigma \int_\Sigma B X_{T,\lambda} b^* (x,y) \, d\mu(y) \, d\mu(x) \\
	&+ \int_\Sigma \int_\Sigma b X_{T,\lambda} B^* (x,y) \, d\mu(y) \, d\mu(x) \\
	&+ \int_\Sigma \int_\Sigma b X_{T,\lambda} b^* (x,y) \, d\mu(y) \, d\mu(x).
\end{align*}
We claim the first three terms on the right are $O_T(\lambda^{-N})$ for $N = 1,2,\ldots$. We will only prove this for the first term -- the argument is the same for the second term and the bound for the third term follows since $X_{T,\lambda}$ is self-adjoint. Interpreting $\mu$ as a distribution on $M$, we write formally
\begin{align}
	\nonumber \int_\Sigma \int_\Sigma &B X_{T,\lambda} B^* (x,y) \, d\mu(y) \, d\mu(x) \\
	\nonumber &= \int_M \int_M X_{T,\lambda}(x,y) B^*\mu(y) \overline{B^*\mu(x)} \, dx \, dy\\
	\label{integral over M} &= \frac{1}{2\pi}\int_{-\infty}^\infty \hat X_{T}(t) e^{-it\lambda} \int_M e^{it\lap g}(B^* \mu)(x) \overline{B^* \mu(x)} \, dx \, dt
\end{align}
Once we show
\begin{equation} \label{wavefronts}
	\WF(e^{it\lap g}B^*\mu) \cap B^*\mu = \emptyset \qquad \text{ for all } t \in \supp \hat X_T,
\end{equation}
the integral over $M$ will be smooth in $t$ by Proposition \ref{B contribution}. Integration by parts in $t$ then gives the desired bound of $O_T(\lambda^{-N})$. To prove \eqref{wavefronts}, suppose $(x,\xi) \in \WF(B^*\mu)$. By part (3) of Lemma \ref{b B decomposition}, $\Phi_t(x,\xi)$ is not in $SN^*_{\supp \mu}\Sigma$ for any $1 \leq |t| \leq T$. By propagation or singularities,
\[
	\WF(e^{it\lap g} B^*\mu) = \Phi_t \WF(B^*\mu),
\]
hence
\begin{equation} \label{wavefronts 2}
	\WF(e^{it\lap g} B^*\mu) \cap \WF(B^* \mu) = \emptyset \qquad \text{ for } 1 \leq |t| \leq T.
\end{equation}
Since the support of $\mu$ has been made small, if there is $(x,\xi) \in SN_{\supp \mu}^*\Sigma$ and some $t > 0$ in the support of $(1 - \beta(t)) \hat \chi(t/T)$ for which $\Phi_t(x,\xi) \in SN_{\supp \mu}^* \Sigma$, then $t \geq 1$ since the diameter of $\supp \mu$ is small and the injectivity radius of $M$ is at least $1$. We now have \eqref{wavefronts}, from which follows \eqref{integral over M} as promised.

What remains is to bound
\begin{equation} \label{b bound}
\left| \int_\Sigma \int_\Sigma b X_{T,\lambda}b^*(x,y) \, d\mu(x) \, d\mu(y) \right| \leq \lambda^{n-d-1} + C_{T,b} \lambda^{n-d-2}.
\end{equation}
We have
\[
	b X_{T,\lambda}b^*(x,y) = \sum_j X_{T}(\lambda_j - \lambda) b e_j(x) \overline{b e_j(y)},
\]
and so we write the integral in \eqref{b bound} as
\begin{align} \label{final bound}
	\sum_j X_T(\lambda_j - \lambda) \left| \int_\Sigma b(x,D)e_j(x) \, d\mu(x) \right|^2.
\end{align}
By the bounds
\[
	|X_T(\tau)| \leq C_{T,N}(1 + |\tau|)^{-N} \qquad \text{ for } N = 1,2,\ldots
\]
and Proposition \ref{b contribution},
\begin{align*}
	&\left| \sum_j X_T(\lambda_j - \lambda) \left| \int_\Sigma b(x,D)e_j(x) \, d\mu(x) \right|^2 \right|\\
	&\hspace{4em} \leq C_T  \lambda^{n-d-1} \int_{\R^d} \int_{S^{n-d-1}} |b(x',\omega)|^2 h(x')^2 \, d\omega \, dx' + C_{T,b} \lambda^{n-d-2}.
\end{align*}
Taking $\epsilon$ in part (2) of Lemma \ref{b B decomposition} small enough so that $\epsilon C_T \leq 1$ yields \eqref{b bound}. This concludes the proof of Proposition \ref{looping prop}.


\newsection{Appendix: Stationary phase tool}

The following tool is a combination of Corollary 1.1.8 with the discussion at the end of Section 1.1 in ~\cite{FIOs}. Let $\phi(x,y)$ be a smooth phase function on $\R^m \times \R^n$ with
\[
	\nabla_y \phi(0,0) = 0 \qquad \text{ and } \qquad \det \nabla_y^2 \phi(0,0) \neq 0,
\]
and let $a(\lambda;x,y)$ be a smooth amplitude with small, adjustable support satisfying
\[
	|\partial_\lambda^j \partial_x^\alpha \partial_y^\beta a(\lambda;x,y)| \leq C_{j,\alpha,\beta} \lambda^{-j} \qquad \text{ for } \lambda \geq 1
\]
for $j = 0,1,2,\ldots$ and multiindices $\alpha$ and $\beta$. $\nabla_y^2 \phi \neq 0$ on a neighborhood of $0$ by continuity. There exists locally a smooth map $x \mapsto y(x)$ whose graph in $\R^m \times \R^n$ contains all points in a neighborhood of $0$ such that $\nabla_y \phi = 0$, by the implicit function theorem. Let $p$ and $q$ be integers denoting the number of positive and negative eigenvalues of $\nabla_y^2 \phi$, respectively, counting multiplicity. By continuity, $p$ and $q$ are constant on a neighborhood of $0$. We adjust the support of $a$ to lie in the intersection of these neighborhoods.

\begin{prop} \label{stationary phase tool}
	Let 
	\[
		I(\lambda; x) = \int_{\R^n} e^{i\lambda \phi(x,y)} a(\lambda; x, y) \, dy
	\]
	with $\phi$ and $a$ as above. Then,
	\begin{align*}
		I(\lambda; x) = (\lambda/2\pi)^{-n/2} &|\det \nabla_y^2 \phi(x,y(x))|^{-1/2} e^{\pi i (p - q)/4} e^{i\lambda \phi(x,y(x))} a(\lambda; x, y(x))\\
		&\hspace{8em}+ \lambda^{-n/2 - 1} e^{i\lambda \phi(x,y(x))} R(\lambda; x) + O(\lambda^{-N})
	\end{align*}
	for $N = 1,2,\ldots$, where $R$ has compact support,
	\[
		|\partial_\lambda^j \partial_x^\alpha R(\lambda; x)| \leq C_{j,\alpha} \lambda^{-j} \qquad \text{ for } \lambda \geq 1,
	\]
	and the $O(\lambda^{-N})$ term is constant in $x$.
\end{prop}

\begin{proof}
	We have
	\begin{equation} \label{eI}
		e^{-i\lambda \phi(x,y(x))} I(\lambda;x) = \int_{\R^n} e^{i\lambda \Phi(x,y)} a(\lambda; x,y) \, dy
	\end{equation}
	where we have set
	\[
		\Phi(x,y) = \phi(x,y) - \phi(x,y(x)).
	\]
	The proof of the Morse-Bott lemma in ~\cite{MorseBott} lets us construct a smooth map $F : \R^m \times \R^n \to \R^n$ such that $y \mapsto F(x,y)$ is a diffeomorphism between neighborhoods of $0$ in $\R^n$ for each $x$, and for which
	\[
		F(x,0) = y(x)
	\]
	and
	\[
		\Phi(x,F(x,y)) = \frac{1}{2} (y_1^2 + \cdots + y_p^2 - y_{p+1}^2 - \cdots - y_{n}^2) = Q(y).
	\]
	Applying a change of variables in $y$ to \eqref{eI} yields
	\begin{align*}
		&= \int_{\R^n} e^{i\lambda Q(y)} a(\lambda; x,F(x,y)) |\det D_yF(x,y)| \, dy \\
		&= \int_{\R^n} e^{i\lambda Q(y)} a(\lambda; x, F(x,0)) |\det D_y F(x,0)| \, dy + \int_{\R^n} e^{i\lambda Q(y)} r(\lambda; x,y) \, dy
	\end{align*}
	where
	\[
		r(\lambda;x,y) = a(\lambda; x,F(x,y)) |\det D_yF(x,y)| - a(\lambda; x, F(x,0)) |\det D_y F(x,0)|.
	\]
	The first term evaluates to
	\[
		(\lambda/2\pi)^{-n/2} |\det \nabla_y^2 \phi(x,y(x))|^{-1/2} e^{\pi i (p - q)/4} a(\lambda; x, y(x))
	\]
	since
	\[
		\int_{-\infty}^\infty e^{i\lambda t^2/2} \, dt = (\lambda/2\pi)^{-1/2} e^{\pi i/4}
	\]
	and
	\[
		|\det \nabla_y^2 \phi(x,y(x))| = |\det D_y F(x,0)|^{-2}.
	\]
	To estimate the second term, we let $\chi$ be a smooth compactly supported cutoff function with $\chi(|y|) = 1$ for all $y \in \supp_y a$. Then,
	\begin{align*}
		\int_{\R^n} &e^{i\lambda Q(y)} r(\lambda; x,y) \, dy\\
			&= \int_{\R^n} e^{i\lambda Q(y)} \chi(|y|) r(\lambda; x,y)  \, dy + \int_{\R^n} e^{i\lambda Q(y)} (1 - \chi(|y|)) r(\lambda; x,y)  \, dy\\
			&= R_1(\lambda;x) + R_2(\lambda;x),
	\end{align*}
	respectively. Since $r(\lambda;x,y)$ vanishes for $y = 0$,
	\[
		|\partial_\lambda^j \partial_x^{\alpha} R_1(\lambda;x)| \leq C_{j,\alpha} \lambda^{-n/2-1-j}
	\]
	by ~\cite[Lemma 1.1.6]{FIOs} applied to the $x$-derivatives of $R_1$. Finally,
	\[
		R_2(\lambda; x) = c\int_{\R^n} e^{i\lambda Q(y)} (1 - \chi(|y|)) \, dy = O(\lambda^{-N})
	\]
	by integration by parts.
\end{proof}


\end{document}